\documentclass{amsart}
\usepackage{amsfonts,amssymb,latexsym,amsthm,amsmath}
\usepackage{euscript}
\newtheorem{theorem}{Theorem}

\newtheorem{lemma}[theorem]{Lemma}

\newtheorem{proposition} [theorem]{Proposition}

\usepackage{tikz}

\usetikzlibrary{matrix}
\usepackage[arrow,matrix,curve]{xy}\SilentMatrices
\def\xyma{\xymatrix@M.7em}

\begin{document}
\title{On nontriviality of homotopy groups of spheres}
\author{Sergei O. Ivanov}
\address{Chebyshev Laboratory, St. Petersburg State University, 14th Line, 29b,
Saint Petersburg, 199178 Russia} \email{ivanov.s.o.1986@gmail.com}

\author{Roman Mikhailov}
\address{Chebyshev Laboratory, St. Petersburg State University, 14th Line, 29b,
Saint Petersburg, 199178 Russia and St. Petersburg Department of
Steklov Mathematical Institute} \email{rmikhailov@mail.ru}
\author{Jie Wu }
\address{Department of Mathematics, National University of Singapore, 2 Science Drive 2
Singapore 117542} \email{matwuj@nus.edu.sg}
\urladdr{www.math.nus.edu.sg/\~{}matwujie}

\thanks{The main result (Theorem~\ref{maintheorem}) is supported by the Russian Science Foundation, grant N 14-21-00035.
The last author (Jie Wu) is also partially supported by the
Singapore Ministry of Education research grant (AcRF Tier 1 WBS
No. R-146-000-190-112) and a grant (No. 11329101) of NSFC of
China.}

\begin{abstract}
For $n\geq 2$, the homotopy groups $\pi_n(S^2)$ are non-zero.
\end{abstract}
 \maketitle

\section{Introduction}

In \cite{Curtis}, E. Curtis proved that $\pi_n(S^4)\neq 0,$ for
all $n\geq 4$. The main method from \cite{Curtis} of proving that
a given element of the homotopy groups of spheres is non-zero is
the analysis of Adams' $d$ and $e$-invariants of the stabilization
of either that element or its Hopf image. This method allowed E.
Curtis to prove that (see \cite{Curtis})
$$
\pi_n(S^2)\neq 0,\ n\not\equiv 1 \mod 8
$$
The same results on non-vanishing terms of the homotopy groups of
spheres were obtained with the help of the composition method by
M. Mimura, M.Mori and N. Oda \cite{MMO}.

Using the methods of the stable homotopy theory, the analysis of
the image of the J-homomorphism and K-theory, it was shown by M.
Mahowald \cite{Mahowald,Mahowald2} and M. Mori \cite{Mori} that
$$
\pi_n(S^5)\neq 0,\ n\geq 5.
$$
From the other hand, since the fourth stable homotopy group of
spheres is zero, one can not get such kind of result for higher
spheres, indeed $\pi_{n+4}(S^n)=0,\ n\geq 6$. The only remaining
case to consider when such kind of phenomena can happen is the
case of $S^2$ and $S^3$. The main result of this paper is the
following
\begin{theorem}\label{maintheorem}
For $n\geq 2$, the homotopy groups $\pi_n(S^2)$ are non-zero.
\end{theorem}
Since $\pi_n(S^3)=\pi_n(S^2),\ n\geq 3,$ the same result follows
for the homotopy groups $\geq 3$ of the 3-sphere.

In the proof of theorem \ref{maintheorem}, we cover the gaps in
dimensions\ $\equiv 1\mod 8$ by showing that, for any odd prime
$p$ and $n\geq 2$,
$$
\mathbb Z/p\subseteq \pi_{(2p-2)n+1}(S^3),
$$
In particular,
$$
\mathbb Z/3\subseteq \pi_{4n+1}(S^3),\ \ \mathbb Z/15\subseteq
\pi_{8n+1}(S^3).
$$

Let $\pi_{k}^n$ denote the $2$-component of $\pi_k(S^n)$. According to~\cite[the Table on page 543] {Curtis}, the $2$-component $\pi_{k}^4\not=0$ for $k>4$. M.
Mahowald \cite[Theorem 1.6]{Mahowald2} and M. Mori \cite[Corollary 5.12 (iv)]{Mori} also proved the stronger statement that $\pi_{k}^5\not=0$ for $k>5$. For the $2$-component $\pi_k^3$ of $\pi_*(S^3)$, Curtis proved that
$\pi_n^3\neq 0,\ n\not\equiv 1,2 \mod 8.$ The non-triviality of these cases can be also read from the fact that the $2$-local $v_1$-periodic homotopy group $v_1^{-1}\pi_{n}^3\not=0$ if and only if $n\not\equiv 1,2\mod{8}$ by ~\cite[Theorem 4.2]{Davis}. For the remaining cases of $\pi_n(S^3)$ with $n\equiv 1,2\mod 8$, notice that the $2$-components of $\pi_9(S^3)$ and $\pi_{10}(S^3)$ both vanish, and so it is necessary to fulfill odd primes for having the non-triviality. Indicated from~\cite[Figure 3.3.18]{Ravenel}, one could have the following conjecture:

\medskip

\noindent\textbf{Conjecture}\footnote{During the private
circulation of this article, Doug Ravenel wrote a comment that it
could be the case that all other $2$-components of $\pi_*(S^3)$
are nontrivial except $\pi_9(S^3)$ and $\pi_{10}(S^3)$.}. The
$2$-component of $\pi_n(S^3)$ is non-trivial for $n>10$.

\medskip

{\it After writing this paper, the authors became aware of the
result from the paper \cite{Gray}. The gaps in dimensions $\equiv
1 \mod 8$ have been covered by a result of Brayton
Gray~\cite[Theorem 12(e)]{Gray} although Theorem~\ref{maintheorem}
was not aware in~\cite{Gray}. We point out that the method of the
present paper for proving Theorem~\ref{maintheorem} is different
from that in~\cite{Gray}.}

\section{Lambda-algebra and Toda elements}

Recall that, for any $k\geq 1$ and an odd prime $p$, the homotopy
groups $\pi_{2(p-1)k+2}(S^3)$ contain non-trivial elements
$\alpha_k(3)$ called the {\it Toda elements.} The elements
$\alpha_k(3)$ have non-zero stable images in $\pi_{2(p-1)k-1}^S$.
We will use the standard notation
$$
\alpha_k(m)=\Sigma^{m-3}(\alpha_k(3))\in \pi_{2(p-1)k+m-1}(S^m),\
m\geq 3.
$$

There exists a $p$-local EHP sequence
$$
J_{p-1}(S^4)\longrightarrow \Omega S^5\buildrel{H_p}\over
\longrightarrow \Omega S^{2p+1},
$$
where $J_{p-1}(S^4)$ is the $(2p-1)$-skeleton of $\Omega S^5$,
which implies the long exact sequence of homotopy groups~\cite[
(2.11), p.103]{Toda}
$$
\dots\longrightarrow
\pi_{n+1}(S^{4p+1})\buildrel{P}\over\longrightarrow
\pi_{n-1}(J_{p-1}(S^4))\buildrel{E}\over\longrightarrow
\pi_n(S^5)\buildrel{H_p}\over\longrightarrow
\pi_n(S^{4p+1})\longrightarrow \dots.
$$

The following statement seems to be known. For example, there is a
discussion of this result at the end of page 535 in
\cite{Bendersky}. However, we were not able to find an explicit
reference to this statement and give here a proof.

\begin{proposition}\label{lemma1} For $k\geq 2$, if the image of the map
$$
H_p: \pi_{2k(p-1)+4}(S^5)\to \pi_{2k(p-1)+4}(S^{4p+1})
$$
contains the element $\alpha_{k-2}(4p+1)$, then $k\equiv 0\mod p$.
\end{proposition}
 Let $p$ be a fixed odd prime number. The
mod-$p$ lambda algebra ${}_{[p]}\Lambda=\Lambda$ is an $\mathbb
F_p$-algebra generated by elements $\lambda_i$ of degree
$2(p-1)i-1$ for $i\geq 1$ and elements $\mu_j$ of degree $2(p-1)j$
for $j\geq 0.$ We will use the following notations for ${\tt
a}(k,j), {\tt b}(k,j)\in \mathbb F_p$ \begin{align*} & {\tt
a}(k,j)=(-1)^{j+1}\binom{(p-1)(k-j)-1}{j},\\ & {\tt
b}(k,j)=(-1)^j\binom{(p-1)(k-j)}{j},\end{align*} and for for
$N(k),N'(k)\in \mathbb Z$:
$$N(k)=\left\lfloor k-\frac{k+1}{p}\right\rfloor,\ N'(k)=\left\lfloor k-\frac{k}{p}\right\rfloor$$
The ideal of relations in $\Lambda$ is generated by the following
relations:
\begin{align*}
&  \lambda_i\lambda_{pi+k}=\sum_{j=0}^{N(k)} {\tt a}(k,j)
\lambda_{i+k-j}\lambda_{pi+j},\ i\geq 1,k\geq 0\\
&  \lambda_i\mu_{pi+k}=\sum_{j=0}^{N(k)} {\tt a}(k,j)
\lambda_{i+k-j}\mu_{pi+j}+\sum_{j=0}^{N'(k)} {\tt b}(k,j)
\mu_{i+k-j}\lambda_{pi+j},\ i\geq 1,k\geq 0\\
& \mu_i\lambda_{pi+k+1}=\sum_{j=0}^{N(k)} {\tt a}(k,j)
\mu_{i+k-j}\lambda_{pi+j+1},\ i\geq 0,k\geq 0\\
& \mu_i\mu_{pi+k+1}=\sum_{j=0}^{N(k)} {\tt a}(k,j)
\mu_{i+k-j}\mu_{pi+j+1},\ i\geq 0,k\geq 0.
\end{align*}
The differential $\partial:\Lambda\to \Lambda$ is given by
\begin{align*}
& \partial\lambda_k =\sum_{j=1}^{N(k)}{\tt
a}(k,j)\lambda_{k-j}\lambda_j,\\ & \partial
\mu_k=\sum_{j=0}^{N(k)} {\tt
a}(k,j)\lambda_{k-j}\mu_j+\sum_{j=1}^{N'(k)}{\tt b}(k,j)
\mu_{k-j}\lambda_{j}.\end{align*}

Further by $\nu_i$ we denote an element of $\{\lambda_i,\mu_i\}.$
A monomial $\nu_{i_1}\dots \nu_{i_l}$ is said to be {\it
admissible} if $i_{k+1}\leq pi_k-1$  whenever
$\nu_{i_k}=\lambda_{i_k}$ and if $i_{k+1}\leq pi_k$ whenever
$\nu_{i_k}=\mu_{i_k}.$ The set of admissible monomials is a basis
of $\Lambda$. The {\it unstable lambda algebra} $\Lambda(n)$ is a
dg-subalgebra of $\Lambda$ generated by admissible elements
$\nu_{i_1}\dots \nu_{i_l}$ such that $i_1\leq n.$ We denote by
$\Lambda(n)_m$ the subspace generated by monomials of degree $m$
in $\Lambda(n)$ and by $\Lambda(n)_{m,l}$ the vector space
generated by monomials of length $l$ in $\Lambda(n)_m.$ Then
$$\Lambda(n)=\bigoplus_{m,l}\Lambda(n)_{m,l}, \hspace{1cm} \Lambda(n)_m=\bigoplus_{l} \Lambda(n)_{m,l}.$$

Consider the left ideal of $\Lambda:$
\begin{equation}
\Lambda\lambda=\sum_i \Lambda \lambda_i.
\end{equation}
The set of all admissible monomials $\nu_{i_1}\dots \nu_{i_l}$
such that $\nu_{i_l}=\lambda_{i_l}$ forms a basis of
$\Lambda\lambda.$ Further we put
$$\Lambda\lambda(n)=\Lambda\lambda\cap \Lambda(n).$$
There exists a spectral sequence which converges to the
$p$-primary components of the homotopy groups of spheres, whose
$E^1$-page is the lambda-algebra and $d^1$-differential is the
differential in the lambda algebra:
$$
E^1(n)=\Lambda\lambda(n)\Rightarrow\ _{(p)}\pi_*(S^{2n+1}).
$$
This is an integral version of the well-known lower central series
spectral sequence of six authors \cite{6_authors}. This spectral
sequence was considered in details in the thesis of D. Leibowitz
\cite{Leibowitz}.

In the language of lambda-algebra, the elements $\alpha_k$ can be
presented as (see, for example 2.9 \cite{Tangora})
$\mu_1^{k-1}\lambda_1$. The map $H_p: \Omega S^5\to \Omega
S^{4p+1}$ induces the map $h_p$ on the level of $E^1$-terms of the
spectral sequence (see page 23 in \cite{Tangora}, and also
\cite{Thompson} and \cite{HM}) with the short exact sequence
$$
0\to \Lambda(1)\oplus \lambda_2\Lambda(5)\longrightarrow
\Lambda(2)\buildrel{h_p}\over\longrightarrow \Lambda(2p)\to 0,
$$
$$
h_p(\mu_1\alpha)=h_p(\lambda_1\alpha)=h_p(\lambda_2\alpha)=0,
$$
for any $\alpha$ and
$$
h_p(\mu_2\alpha)=\alpha\in \Lambda(4p+1).
$$

\begin{lemma}
The linear map
$$d_1:{\sf span}(\mu_1^k\lambda_2,\{\mu_1^{k-i}\mu_2\mu_1^{i-1}\lambda_1\}_{i=1}^k)\longrightarrow {\sf span}(\{\mu_1^{k-i}\lambda_1\mu_1^i\lambda_1\}_{i=0}^k)$$
is an isomorphism if and only if $k+2 \not\equiv 0 ({\rm mod}\:
p).$
\end{lemma}

\begin{proof}
Using the definition of $d_1=:d$ we get $$d(\lambda_1)=0, \  \
d(\mu_1)=-\lambda_1\mu_0, \  \ d(\lambda_2)=-2\lambda_1^2, \  \
d(\mu_2)=-\lambda_2\mu_0-2\lambda_1\mu_1+\mu_1\lambda_1.$$Using
the relations $\mu_0\mu_1=0= \mu_0\lambda_1$ and
$\mu_0\lambda_2=-\mu_1\lambda_1,$ $\mu_0\mu_2=-\mu_1\mu_1,$ it is
easy to compute that \begin{align*} &
d(\mu_2\lambda_1)=-2\lambda_1\mu_1\lambda_1+\mu_1\lambda_1^2,\\
& d(\mu_1\lambda_2)=\lambda_1\mu_1\lambda_1-2\mu_1\lambda_1^2,\\
&
d(\mu_1\mu_2\lambda_1)=\lambda_1\mu_1^2\lambda_1-2\mu_1\lambda_1\mu_1\lambda_1+\mu_1^2\lambda_1^2.
\end{align*} Moreover, we obtain $d(\mu_1)\mu_1=0$ and
$d(\mu_1)\lambda_1=0.$ It follows that
\begin{align*}
&
d(\mu_1^k\lambda_2)=\mu_1^{k-1}d(\mu_1\lambda_2)=\mu_1^{k-1}\lambda_1\mu_1\lambda_1-2\mu_1^k\lambda_1^2,\\
&
d(\mu_1^{k-1}\mu_2\lambda_1)=\mu_1^{k-2}d(\mu_1\mu_2\lambda_1)=\mu_1^{k-2}\lambda_1\mu_1^2\lambda_1-2\mu_1^{k-1}\lambda_1\mu_1\lambda_1+\mu_1^k\lambda_1^2.\\
&
d(\mu_1^{k-i-1}\mu_2\mu_1^i\lambda_1)=\mu_1^{k-i-2}d(\mu_1)\mu_2\mu_1^i\lambda_1+\mu_1^{k-i-1}d(\mu_2)\mu_1^i\lambda_1=\\
& \ \ \ \
=\mu_1^{k-i-2}\lambda_1\mu_1^{i+2}\lambda_1-2\mu_1^{k-i-1}\lambda_1\mu_1^{i+1}\lambda_1+\mu_1^{k-i}\lambda_1\mu_1^{i}\lambda_1
\end{align*}
for $1\leq i\leq k-2$ and
$$d(\mu_2\mu_1^{k-1}\lambda_1)=d(\mu_2)\mu_1^{k-1}\lambda_1=-2\lambda_1\mu_1^{k}\lambda_1+\mu_1\lambda_1\mu_1^{k-1}\lambda_1.$$

If we denote $v_i:=\mu_1^{k-i}\lambda_1\mu_1^i\lambda_1$ for
$0\leq i\leq k,$ and $u_i=\mu_1^{k-i}\mu_2\mu_1^{i-1}\lambda_1$
for $1\leq i\leq k$ and $u_0=\mu_1^k\lambda_2,$ then
$$d(u_0)=v_1-2v_0, \ \ d(u_i)=v_{i+1}-2v_{i}+v_{i-1}, \ \ d(u_{k})=-2v_k+v_{k-1}.$$
The matrix corresponding to this linear map is the following
matrix
$$\left(\begin{array}{rrrrrr}
-2 & 1 & 0 & \ 0&  \dots & \ 0\\
1 & -2 & 1& 0 & \dots & 0\\
0 & 1 & -2 & 1 & \ddots & 0\\
\vdots &\ddots &\ddots & \ddots&\ddots & \vdots\\
0 & \dots  & 0 & 0 & 1 & -2\\
\end{array}\right).$$
It is easy to check by induction that its determinant is equal to
$(-1)^{k+1}(k+2).$ It follows that $d:{\sf span}(u_0,\dots,u_k)\to
{\sf span}(v_0,\dots,v_k)$ is an isomorphism  if and only if $ k+2
\not\equiv 0 ({\rm mod}\: p).$
\end{proof}

Now we are ready to prove proposition \ref{lemma1}.
\begin{proof}[Proof of proposition \ref{lemma1}]
Indeed, if the Toda element $\alpha_{k-2}=\mu_1^{k-3}\lambda_1$
lies in the $H_p$ image, then there must be some term on
$E^2$-page like $C\mu_2\mu_1^{k-3}\lambda_1+\sum\dots,\
C\not\equiv 0\mod p,$ which maps onto $\alpha_{k-2}$ by $h_p$.
However, by lemma 2, this is possible only in the case $k\equiv
0\mod p$, in all other cases the corresponding $E^2$-term of the
spectral sequence for $S^5$ is zero.\end{proof}

\section{Proof of theorem \ref{maintheorem}}
For the proof of theorem \ref{maintheorem}, we will use the
following classical results in homotopy theory.

\vspace{0.5cm}\noindent (1)~\cite{Adams} or~\cite[(4.3),
p.112]{Toda}. The element $\alpha_k\in \pi_{2(p-1)k-1}^S$ is not
divisible by $p$ for $k\not\equiv 0\mod p$.

\vspace{0.5cm}\noindent (2). Let $p>2$. By the classical work of
Cohen, Moore and Neisendorfer~\cite{CMN},  there exists a map
$\pi\colon \Omega^2S^{2n+1}\to S^{2n-1}$ such that the composite
$$
\Omega^2S^{2n+1}\buildrel{\pi}\over \longrightarrow  S^{2n-1}\buildrel{\Sigma^2}\over\longrightarrow  \Omega^2S^{2n+1}
$$
is homotopic to the $p$-th power map $p\colon \Omega^2S^{2n+1}\to \Omega^2S^{2n+1}$, where the case $p=3$ is given in~\cite[Theorem 4.1]{Neisendorfer}. Following the notation in~\cite{CMN}, let $D(n)$ be the homotopy fibre of $\pi\colon \Omega^2S^{2n+1}\to S^{2n-1}$. According to~\cite[Section 6]{CMN}, $D( p)\simeq \Omega^2S^3\langle 3\rangle$ and so there is a fibre sequence
$$
\Omega S^{2p-1} \buildrel{\tau} \over\longrightarrow \Omega^2S^3\langle 3\rangle\buildrel{\theta}\over \longrightarrow  \Omega^2 S^{2p+1}\buildrel{\pi}\over \longrightarrow S^{2p-1}
$$
that implies a long exact sequence
$$
\dots\longrightarrow \pi_{n+1}(S^{2p+1})\buildrel{\pi_*}\over\longrightarrow  \pi_{n-1}(S^{2p-1})\buildrel{\tau_*}\over\longrightarrow
\pi_n(S^3)\buildrel{\theta_*}\over\longrightarrow  \pi_n(S^{2p+1})\longrightarrow \dots.
$$
with the property that, for every $i$, the composition
$$
\pi_{i+2}(S^{2p+1})\buildrel{\pi_*}\over\longrightarrow\pi_i(S^{2p-1})\buildrel{\Sigma^2}\over\longrightarrow\pi_{i+2}(S^{2p+1})
$$
is the multiplication by $p$.

\vspace{.5cm}\noindent (3)~\cite[(2.12), p. 104]{Toda}. For $m\geq
2,$ denote by $Q_2^{2m-1}$, the homotopy fibre of the double
suspension map $S^{2m-1}\to \Omega^2S^{2m+1}$. We will use the
notation from \cite{Toda}. The natural map $Q_2^{2m-1}\to
S^{2m-1}$ induces the map on homotopy groups $p_*$. There is a
natural map
$$
I: \pi_i(Q_2^{2m-1})\to \pi_{i+3}(S^{2mp+1}),
$$
such that the composition
$$
\pi_{i+3}(S^{2m+1})\to
\pi_i(Q_2^{2m-1})\buildrel{I}\over\longrightarrow
\pi_{i+3}(S^{2mp+1})
$$
is the Hopf map $H_p$.

\vspace{.5cm} For a given $k\not\equiv 1\mod p$, consider the
element
$$\alpha_{k-1}\in
\pi_{2(p-1)(k-1)+2p-1}(S^{2p-1})=\pi_{2(p-1)k}(S^{2p-1}).$$
Suppose that $\alpha_{k-1}(2p-1)\in im\{\pi_*:
\pi_{2(p-1)k+2}(S^{2p+1})\to \pi_{2(p-1)k}(S^{2p-1})\}.$ Then the
element $\Sigma^2\alpha_{k-1}(2p-1)=\alpha_{k-1}(2p+1)$ is
$p$-divisible by (2), hence its stable image is $p$-divisible. But
this is not possible by (1). We conclude that
$$\tau_*(\alpha_{k-1}(2p-1))\not=0$$
by the long exact sequence in (2),  and so
 $$(A)\ \ \ \ \ \ \ \
\ \ \ \ \ \  \mathbb Z/p\subseteq \pi_{2(p-1)k+1}(S^3),\
k\not\equiv 1 \mod p$$

Now we recall the following statement of Toda (Theorem 5.2 (ii)
\cite{Toda}, case $m=1$). For $k\geq 2$, there exist an element
$$
\gamma'\in \pi_{2p+2k(p-1)-1}(Q_2^3)=\pi_{2(p-1)(k+1)+1}(Q_2^3)
$$
such that
$$
I(\gamma')=\alpha_{k-1}(4p+1)\in \pi_{4p+2(k-1)(p-1)}(S^{4p+1}).
$$
Here $I: \pi_{2p+2k(p-1)-1}(Q_2^3)\to
\pi_{2p+2k(p-1)+2}(S^{4p+1}).$ Suppose that $p_*(\gamma')=0$, then
$$
\gamma'\in im\{H^{(2)}: \pi_{2(p-1)(k+1)+4}(S^5) \to
\pi_{2(p-1)(k+1)+1}(Q_2^3).\}
$$
In this case, we get
$$
\alpha_{k-1}(4p+1)\in im\{H_p: \pi_{2(k+1)(p-1)+4}(S^5)\to
\pi_{2(k+1)(p-1)+4}(S^{4p+1})\}
$$
This is possible only for $k+1\equiv 0\mod p$ by proposition
\ref{lemma1}. For $k+1\not\equiv 0\mod p$, we get $0\neq
p_*(\gamma')\in\pi_{2(p-1)(k+1)+1}(S^3)$. Therefore,
$$
(B)\ \ \ \ \ \ \ \ \ \ \ \ \ \ \mathbb Z/p\subseteq
\pi_{2(p-1)k+1}(S^3),\ k\not\equiv 0\mod p.
$$
The statements $(A)$ and $(B)$ together give the needed statement:
$$
\mathbb Z/p\subseteq \pi_{2(p-1)k+1}(S^3),\ k\geq 1.
$$
Theorem \ref{maintheorem} now follows, since all dimensions
$\equiv 1\mod 8$ are covered, moreover there is a $\mathbb
Z/15$-summand in homotopy groups $\pi_{8l+1}(S^2),\ l\geq 2$.\ \ \
$\Box$

\vspace{.5cm} As a final remark we observe that homotopy groups of
$S^2$ in certain dimension $\equiv 1 \mod 8$ can be covered in
another way. For that, we recall the results from \cite{Toda} and
\cite{Mori}.

\vspace{.5cm}\noindent (iv) (Lemma 15.3 (i),\cite{Toda3}) Let
$y\in \pi_i(S^{2p-1})$ be an element of order $p$. There exists an
element $a\in \pi_{i+2}(S^3)$, such that
$$
H_p(a)=x\Sigma^2y\in \pi_{i+2}(S^{2p+1})
$$
for some $x\not\equiv 0\mod p$.

\vspace{.5cm}\noindent (v) For $f\geq 0$, there is a family of
elements $\alpha_i^{(f)}\in \pi_{2i(p-1)p^f+2f+2}(S^{2f+3})$ of
order $p^f$, which have non-zero stable image in
$\pi_{2i(p-1)p^f-1}^S$. The $e$-invariants of these elements are
the following: $e_C(\alpha_i^{(f)})=-p^{-f-1}$.

\vspace{.5cm}\noindent (vi) (Lemma 4.1, \cite{Mori}) Let $f,g\geq
0,\ i,j\geq 1$ and
\begin{align*}
& \alpha: S^{2n+2i(p-1)p^f-1}\to S^{2n},\\
& \beta: S^{2n+2i(p-1)p^f+2j(p-1)p^g-2}\to S^{2n+2i(p-1)p^f-1}.
\end{align*}
Assume that $e_C(\alpha)e_C(\beta)=p^{-u}$ and
\begin{align*}
& \nu_p(j)+g+1<u\leq \nu_p(i)+f+1+i(p-1)p^f,\\
& u+\nu_p(ip^f+jp^g)-\nu_p(i)-f-i(p-1)p^f\leq
n<u+\nu_p(ip^f+jp^g)-\nu_p(j)-g,
\end{align*}
then $\alpha\circ \beta$ non zero.

\vspace{.5cm} Now we will show that, for any $k\geq 1$, there is a
non-zero $p$-torsion element in $\pi_{2(p-1)(p^pk+1)+1}(S^3)$. For
that, consider the case $g=0, f=p-2, i=p^2k-1, j=p^{p-2}$. By
(vi), we see that, $\alpha_i^{(p-2)}\circ \alpha_j^{(0)}$ is a
non-zero element in homotopy group
$\pi_{2(p-1)p^pk+2p-1}(S^{2p+1})$ which equals to the image of the
double suspension of an element of order $p$ from
$\pi_{2(p-1)p^pk+2p-3}(S^{2p-1})$. Hence, by (iv), there is an
element in
$\pi_{2(p-1)p^pk+2p-1}(S^3)=\pi_{2(p-1)(p^pk+1)+1}(S^3)$ whose
$H_p$-image gives a non-zero multiple of this element.

 \vspace{.5cm}\noindent {\it
Acknowledgements.} The authors would like to thank F. Petrov for discussions
related to the subject of the paper. The authors also wish to thank Martin Bendersky, Fred Cohen, Don Davis, Brayton Gray, Haynes Miller and Doug Ravenel most warmly for their important comments during the private circulation of this article.

\end{document}